\documentclass{amsart}
\pdfoutput=1
\usepackage[utf8]{inputenc} % allow utf-8 input
\usepackage[T1]{fontenc}    % use 8-bit T1 fonts
\usepackage{lmodern}
\usepackage{hyperref}       % hyperlinks
\usepackage{url}            % simple URL typesetting
\usepackage{booktabs}       % professional-quality tables
\usepackage{amsfonts,amsmath,amsthm,amssymb}       % blackboard math symbols
\usepackage{nicefrac}       % compact symbols for 1/2, etc.
\usepackage{microtype}      % microtypography
\usepackage{xcolor}
\usepackage[margin=1.3in]{geometry}

\theoremstyle{plain}
\newtheorem{theorem}{Theorem}[section]

\newtheorem{lemma}[theorem]{Lemma}
\newtheorem{theoremA}{Theorem}

\newtheorem{corollaryB}[theoremA]{Corollary}

\theoremstyle{definition}
\newtheorem{definition}{Definition}[section]

\theoremstyle{remark}
\newtheorem{remark}{Remark}[section]

\newcommand{\RCD}{\mathsf{RCD}}
\newcommand{\CD}{\mathsf{CD}}
\newcommand{\CDloc}{\protect{\mathsf{CD}_\text{loc}}}
\renewcommand{\d}{\mathsf{d}}
\newcommand{\m}{\mathfrak{m}}
\newcommand{\Ric}{\mathrm{Ric}}
\newcommand{\ov}{\overline}

\newcommand{\R}{\mathbb{R}}
\newcommand{\N}{\mathbb{N}}
\renewcommand{\P}{\mathcal{P}}
\newcommand{\F}{\mathcal{F}}
\newcommand{\s}{\mathfrak{s}}
\DeclareMathOperator{\Geo}{Geo}
\DeclareMathOperator{\OptGeo}{OptGeo}
\DeclareMathOperator{\supp}{supp}

\title[$\CD$ spaces nonnegatively curved outside a compact set]{On $\CD$ spaces with nonnegative curvature outside a compact set}

\author{Mauricio Che}
\address[M. Che]{Durham University, Durham, United Kingdom}
\email{\href{mailto:mauricio.a.che-moguel@durham.ac.uk}{mauricio.a.che-moguel@durham.ac.uk}}   
\author{Jes\'us N\'u\~{n}ez-Zimbr\'on}
\address[J. N\'u\~{n}ez-Zimbr\'on]{Centro de Investigaci\'on en Matem\'aticas, Guanajuato, Mexico}
\email{\href{mailto:jesus.nunez@cimat.mx}{jesus.nunez@cimat.mx}}

\thanks{M. Che was partially supported by CONACYT-Doctoral scholarship no. 769708}
\keywords{$\CD$ spaces, non-branching spaces, ball covering property, number of ends}
\date{\today}
\subjclass[2010]{53C23, 53C21}

\begin{document}
\begin{abstract}
In this paper we adapt work of Z.-D. Liu to prove a ball covering property for non-branching $\CD$ spaces with nonnegative curvature outside a compact set. As a consequence we obtain uniform bounds on the number of ends of such spaces.
\end{abstract}

\maketitle

\section{Introduction}
In \cite{liu1,liu2}, Z.-D. Liu proved that Riemannian manifolds with nonnegative Ricci curvature outside a compact set satisfy the following \emph{ball covering property}. In the following we denote the metric ball of radius $r$ centered at $p\in M$ by $B_r(p)$ and the closed metric ball with the same radius and center by $\ov{B}_{r}(p)$. 

\begin{theorem}\label{thm:liu}
Let $M^n$ be a complete Riemannian manifold with nonnegative Ricci curvature outside a compact set $B$. Assume that $\Ric_M \geq (n-1)H$ and that  $B\subset B_{D_0}(p_0)$ for some $p_0\in M$ and $D_0>0$. Then for any $\mu>0$ there exists $C=C(n,HD_0^2,\mu)>0$ such that for any $r > 0$, the following property is satisfied:  If $S \subset \ov{B}_{r}(p_0)$, there exist $p_1,\dots, p_k \in S$ with $k \leq C$ such that
\[
S\subset \bigcup_{j=1}^{k}B_{\mu\cdot r}(p_j). 
\]
\end{theorem}

We state and prove this result in the more general context of non-branching metric measure spaces satisfying the curvature-dimension condition introduced by Lott-Sturm-Villani \cite{lott-villani,sturm1,sturm2} (see the section on preliminaries below for the definitions). This class of spaces contains the class of $\RCD$ spaces, as it was recently shown that these are non-branching (see \cite[Theorem 1.3]{deng2020holder}), so, a fortiori, it also includes Alexandrov spaces \cite{Petrunin2011, Zhang-Zhu2010} and weighted Riemannian manifolds. More precisely, we prove the following theorem.

\begin{theoremA}\label{thm:principal}
Let $(X,\d,\m)$ be a non-branching metric measure space satisfying the $\CD(K,N)$ condition for some $N> 1$ and $K\in \R$. Assume that $B$ is a compact subset of $X$ with $B\subset B_{D_0}(p_0)$ for some $p_0\in M$ and $D_0>0$ and such that the $\CDloc(0,N)$ condition is satisfied on $X\setminus B$ (see Definition \ref{def:cd outside B}). Then for any $\mu > 0$, there exists $C = C(N, KD_0^2,\mu) > 0$ such that for any $r > 0$, the following property is satisfied: If $S \subset \ov{B}_{r}(p_0)$, there exist $p_1,\dots, p_k \in S$ with $k \leq C$ and such that
\[
S\subset \bigcup_{j=1}^{k}B_{\mu\cdot r}(p_j). 
\]
\end{theoremA}

The proof follows the arguments of \cite{liu1,liu2} almost verbatim, albeit with some needed adaptations to account for the more general hypotheses. The main tools we need are a version of the local-to-global theorem for the $\mathsf{CD}$ condition (see Lemma \ref{lem:key lemma}) and a Bishop-Gromov inequality for certain star-shaped sets (see Theorem \ref{thm:bishop-gromov}). In general one can prove that $\CD(K,N)$ spaces support a Bishop-Gromov inequality for star-shaped sets following the proof of \cite[Theorem 2.3]{sturm2}, just as is done in \cite[Proposition 3.5]{cavalleti-mondino-lorentz} to get a timelike Bishop-Gromov inequality in the context of Lorentzian synthetic spaces.  However, it is important to notice that Lemma \ref{thm:bishop-gromov} is not a direct consequence of this fact. Namely, since the $\CD(K,N)$ condition implies a Bishop-Gromov inequality  with parameters $K,N$ and we are interested in the corresponding inequality with parameters $0, N$, we need to follow the original proof of the Bishop-Gromov inequality in \cite{sturm2} and make sure that all optimal transports involved remain in the region where $\CDloc(0,N)$ holds.

Finally, a direct consequence is that spaces satisfying the hypotheses of Theorem \ref{thm:principal} have a uniformly bounded number of ends (see Definition \ref{def:ends}). 

\begin{corollaryB}\label{cor:principal}
Let $(X,\d,\m)$ be a metric measure space satisfying the $\CD(K,N)$ with $N\geq 1$ and $K\in \R$. Assume that $B$ is a compact subset of $X$ with $B\subset B_{D_0}(p_0)$ for some $p_0\in M$ and $D_0>0$ and such that the $\CDloc(0,N)$ condition is satisfied on $X\setminus B$. Then there exists $C=C(N,KD_0^2)>0$ such that $(X,\d)$ has at most $C$ ends.
\end{corollaryB}

In \cite{cai}  a result bounding the number of ends of manifolds with nonnegative Ricci curvature outside a compact set was obtained using different techniques. This argument was extended in \cite{jywu} to the case of  smooth metric measure spaces with nonnegative Brakry-\'{E}mery Ricci curvature outside of a compact set. A related result bounding the number of ends of Alexandrov spaces with nonnegative sectional curvature outside a compact set was obtained in \cite{koh}. More recently, a result bounding the number of ends of $\RCD(0,N)$ spaces was obtained in \cite{gigli2021monotonicity}.\\

\textbf{Acknowledgements.} The authors wish to thank Fabio Cavalletti, Fernando Galaz-Garc\'{i}a, Nicola Gigli and Guofang Wei for very useful communications.

\section{Preliminaries}
\label{sec:preliminaries}
In this section we provide a brief overview of the definitions and results we will need to prove Theorem \ref{thm:principal}. Throughout the article we consider complete and geodesic metric measure spaces $(X,\d,\m)$ such that $\m$ is finite on bounded sets and $\supp(\m)=X$. We begin by recalling the definition of the so-called Wasserstein space. 

\begin{definition}
Let $\P(X,\d)$ be the set of Borel probability measures on $X$ and $\P_2(X,\d,\m)\subset \P(X,\d)$ the space of those probability measures that are absolutely continuous with respect to $\m$ and have finite second moment, i.e. for some (and therefore for any) $x_0\in X$ the following holds:
\[
\int d^2(x,x_0)\ d\m(x) < \infty.
\]
This set $\P_2(X,\d,\m)$ is endowed with the $2$-Wasserstein metric
\[
W_2(\mu_0,\mu_1) = \inf \left(\int_{X\times X} \d^2(x_0,x_1)\ d\pi(x_0,x_1)\right)^{1/2}
\] where the infimum is taken over all couplings $\pi \in\P(X\times X)$ from $\mu_0$ to $\mu_1$, i.e. probability measures on $X\times X$ having first and second marginals equal to $\mu_0$ and $\mu_1$ respectively.
\end{definition}

\begin{remark}
It turns out that $(\P_2(X),W_2)$ is also a complete separable geodesic space. Moreover, in this case, the distance $W_2(\mu_0, \mu_1)$ can be characterized as
\[
W_2^2 (\mu_0,\mu_1) = \min_\pi \int \int^1_0 |\gamma_t|^2\ dt\ d\pi(\gamma),
\]
where the minimum is taken among all $\pi \in \P(C([0, 1], X))$ such that $(e_i)_\#\pi = \mu_i$, $i=0,1$. Here $e_t$ denotes the usual evaluation map at time $t$.  The set of minimizers is denoted by $\OptGeo(\mu_0,\mu_1)$, and minimizers, which are always supported in $\Geo(X)$ (the set of geodesics of $(X,\d)$), are called \emph{optimal plans}. It is known that $(\mu_t)_{t\in [0,1]}$ is a geodesic connecting $\mu_0$ to $\mu_1$ if and only if there exists $\pi \in \OptGeo(\mu_0,\mu_1)$ such that $\mu_t = (e_t)_\#\pi$ (see \cite{gigli-ambrosio}).
\end{remark}

In order to recall the definition of the $\mathsf{CD}$ condition, we now recall the volume distortion coefficients:
\[
    \s_\kappa(\theta) = \begin{cases}
    \frac{1}{\sqrt{\kappa}}\sin(\sqrt{\kappa}\theta) & \text{if } \kappa > 0,\\
    \theta & \text{if }\kappa = 0,\\
    \frac{1}{\sqrt{\kappa}}\sinh(\sqrt{-\kappa}\theta) & \text{if } \kappa < 0,
    \end{cases}
\]

\[    
    \sigma^{(t)}_{K,N}(\theta) = 
    \begin{cases}
    +\infty  &   \text{if } K\theta^2 \geq N\pi^2,\\
    t   &   \text{if }K\theta^2 = 0\text{ or }K\theta^2 < 0\text{ and }N=0,\\
    \frac{\s_{K/N}(t\theta)}{\s_{K/N}(\theta)}  &   \text{if }K\theta^2<N\pi^2 \text{ and } K\theta^2\neq 0,
    \end{cases}
\]

\[
    \tau^{(t)}_{K,N}(\theta) = t^{1/N} \sigma^{(t)}_{K,N-1}(\theta)^{1-1/N}.
\]

\begin{definition}\label{def:CD}
Given parameters $K\in \R$ and $N\geq 1$, we say that $(X,\d,\m)$ satisfies the \emph{$\CD(K,N)$ condition} if for any $\mu_0,\mu_1\in \P_2(X,\d,\m)$ there exists an optimal plan $\pi\in\OptGeo(\mu_0,\mu_1)$ such that for any $t\in [0,1]$ and any $N'\geq N$
\begin{equation}\label{eq:cd condition}
\int\rho_t^{1-1/N'}(x)\ d\m(x) \geq \int\tau^{(1-t)}_{K,N'}(d(\gamma_0,\gamma_1))\rho_0^{-1/N'}(\gamma_0) +\tau^{(t)}_{K,N'}(d(\gamma_0,\gamma_1))\rho_1^{-1/N'}(\gamma_1)\ d\pi(\gamma)
\end{equation}
where $\rho_t$ is the density of the absolutely continuous part of $(e_t)_\#\pi$ with respect to $\m$.
\end{definition}

Let us also recall the definition of the reduced curvature-dimension condition $\mathsf{CD}^*$ due to Bacher-Sturm \cite{bacher-sturm}.

\begin{definition}\label{def:CDreduced}
Given parameters $K\in \R$ and $N\geq 1$, we say that $(X,\d,\m)$ satisfies the $\CD^*(K,N)$ if for any $\mu_0,\mu_1\in \P_2(X,\d,\m)$ there exists an optimal plan $\pi\in\OptGeo(\mu_0,\mu_1)$ such that for any $t\in [0,1]$ and any $N'\geq N$
\begin{equation}\label{eq:cd reduced}
\int\rho_t^{1-1/N'}(x)\ d\m(x) \geq \int\sigma^{(1-t)}_{K,N'}(d(\gamma_0,\gamma_1))\rho_0^{-1/N'}(\gamma_0) +\sigma^{(t)}_{K,N'}(d(\gamma_0,\gamma_1))\rho_1^{-1/N'}(\gamma_1)\ d\pi(\gamma)
\end{equation}
where $\rho_t$ is the density of the absolute continuous part of $(e_t)_\#\pi$ with respect to $\m$.
\end{definition}

\begin{definition}\label{def:CDloc}
Given parameters $K\in \R$ and $N>1$, we say that $(X,\d,\m)$ satisfies the $\CDloc(K,N)$ condition if each point $x\in X$ has a neighbourhood $M(x)$ such that for each $\mu_0,\mu_1\in\P_2(X,\d,\m)$ supported in $M(x)$ there exists an optimal plan $\pi\in \OptGeo(\mu_0,\mu_1)$ satisfying \eqref{eq:cd condition} for all $t\in[0,1]$ and $N'\geq N$.
\end{definition}

We will assume that Definition \ref{def:CDloc} holds outside a compact set $B\subset X$ in the following sense. 
\begin{definition}\label{def:cd outside B}
We say that $\CDloc(K,N)$ condition holds in an open set $\Omega\subset X$ if each point $x\in \Omega$ has a neighbourhood $M(x)\subset \Omega$ such that for each $\mu_0,\mu_1\in\P_2(X,\d,\m)$ supported in $M(x)$ there exists an optimal plan $\pi\in \OptGeo(\mu_0,\mu_1)$ satisfying \eqref{eq:cd condition} for all $t\in[0,1]$ and $N'\geq N$.
\end{definition}

From this point on, we will also assume that $(X,\d,\m)$ is non-branching in the following sense.
\begin{definition}\label{def:non branching}
We say that a metric space $(X,\d)$ is \emph{non-branching} if whenever we have a tuple $(z,x_0,x_1,x_2)$ such that $z$ is a midpoint of $x_0,\ x_1$ and of $x_0,\ x_2$, this implies that $x_1=x_2$.
\end{definition}

In \cite{gigli-optimal} it was proved that non-branching $\CD$ spaces have unique optimal plans in the following sense.
\begin{theorem}\label{thm:optimal maps non branching}
Let $(X,\d,\m)$ be a complete, separable and non-branching $\CD(K,N)$-space for some $K\in \R$ and $N\geq 1$. Then for any $\mu_0,\mu_1\in \P_2(X,\d,\m)$ there is a unique optimal plan $\pi \in \OptGeo(\mu_0,\mu_1)$ and this $\pi$ is induced by a map, i.e. there exists a $\mu_0$-measurable map $F:X\to \Geo(X)$ such that $\pi = F_\#\mu_0$. 
\end{theorem}

In \cite{bacher-sturm} it was proved that the $\CDloc(K,N)$ condition implies the reduced curvature-dimension condition $\CD^*(K,N)$. In a similar fashion, and emulating the arguments in \cite{bacher-sturm}, we can prove the following key result that will allow us to generalize Theorem \ref{thm:liu}. Below, we let $\P_\infty(X,\d,\m)$ denote the space of probability measures which are absolutely continuous with respect to $\m$ and have bounded support. For the next lemma, note that we cannot directly apply the local-to-global property since we are using the restricted metric on $X\setminus B$ and this might not be a geodesic space unless $X\setminus B$ is geodesically convex, for example. However, the proof of \cite[Theorem 5.1]{bacher-sturm} applies verbatim as we are assuming that all the measures involved are connected by a geodesic in $\P_\infty(X,\d,\m)$ {\footnote{In fact the proof of \cite[Claim 5.2]{bacher-sturm} simplifies in our case as by Theorem \ref{thm:optimal maps non branching} the geodesic joining $\mu_0$ and $\mu_1$ is unique, so there is no need to construct the sequence $\Gamma^{(i)}$}}.

\begin{lemma}\label{lem:key lemma}
Assuming the hypotheses from Theorem \ref{thm:principal}, if $\mu_0,\mu_1\in\P_\infty(X,\d,\m)$ are supported on $X\setminus B$ and the optimal plan $\pi\in\OptGeo(\mu_0,\mu_1)$ given by Theorem \ref{thm:optimal maps non branching} is supported in geodesics contained in $X\setminus B$, then $\pi$ satisfies the condition \eqref{eq:cd reduced} for $K=0$ and for all $t\in[0,1]$ and $N'\geq N$.
\end{lemma}

Now we can follow the arguments in \cite{sturm2} to prove a generalized Bishop-Gromov result for star-shaped sets outside a compact set. To this end, we need the following modified version of the Brunn-Minkowski inequality.

\begin{theorem}[Brunn–Minkowski inequality]\label{thm:brunn-minkowski}
Let $(X,\d,\m)$ be a metric measure space and $\Omega\subset X$ an open set such that $\CDloc(K,N)$ condition holds in $\Omega$. Then, for all measurable sets $A_0,A_1 \subset \Omega$ such that $\m(A_0)\m(A_1)>0$ and $A_t\subset \Omega$ for all $t\in [0,1]$,
\begin{equation}
\m(A_t)^{1/N'}\geq \sigma^{(1-t)}_{K,N'}(\theta)\m(A_0)^{1/N'}+\sigma^{(t)}_{K,N'}(\theta)\m(A_1)^{1/N'},
\end{equation}
holds for all $t\in [0,1]$ and all $N'\geq N$, where $A_t$ denotes the set of points which divide geodesics starting in $A_0$ and ending in $A_1$ with ratio $t:(1-t)$ and where $\theta$ denotes the minimal/maximal length of such geodesics, that is,
\[
A_t := \{y\in X : \exists (x_0 ,x_1 )\in A_0 \times A_1 \text{ s.t. }\d(y,x_0)=t\d(x_0 ,x_1 ),\ \d(y,x_1 )=(1-t)\d(x_0 ,x_1 )\}
\]
and
\[
\theta:=\begin{cases}
\inf_{x_0 \in A_0,x_1\in A_1} \d(x_0 ,x_1 ), & \text{if } K \geq 0,\\
\sup_{x_0 \in A_0,x_1\in A_1} \d(x_0 ,x_1 ), & \text{if } K < 0.
\end{cases}
\]
In particular, if $K\leq 0$ then
\[
\m(A_t)^{1/N'}\geq (1-t)\m(A_0)^{1/N'}+t\m(A_1)^{1/N'}
\]
\end{theorem}
\begin{proof} 
Assuming that $0<\m(A_0)\m(A_1)<\infty$, and thanks to Theorem \ref{lem:key lemma}, we can apply $\CD^*(K,N)$ to $\mu_i := (1/\m(A_i))1_{A_i}\m$ for $i=0,1$ and proceed as in \cite[Proposition 2.1]{sturm2}, just replacing the coefficients $\tau^{(t)}_{K,N'}(\cdot)$ by $\sigma^{(t)}_{K,N'}(\cdot)$. The general case follows by approximation of $A_i$ by sets of finite volume.
\end{proof}

Recall that, given a metric space $(X,\d)$ and $x\in X$, it is said that $W_x\subset X$ is \emph{star-shaped} at $x$ if $x\in W_x$ and for any $y\in W_x$ not in the cut-locus of $x$, the minimal geodesic joining $x$ and $y$ is contained in $W_x$. In that case, we set
\[
v(W_x,r)=\m(\overline{B}_r(x)\cap W_x) \quad{and}\quad
s(W_x,r)=\limsup_{\eta \to 0} \m((\overline{B}_{r+\eta}(x)\setminus B_r(x))\cap  W_x)
\]

\begin{theorem}[Bishop-Gromov theorem for star-shaped sets]\label{thm:bishop-gromov}
Let $(X,\d,\m)$ be a metric measure space and $\Omega\subset X$  an open set such that $\CDloc(K,N)$ condition holds in $\Omega$. Let $x\in\Omega$ and $W_x\subset \Omega$ be a star-shaped set at $x$ such that for some $\epsilon_0>0$ every geodesic connecting points in $B_{\epsilon_0}(x)$ with points in $W_x$ is contained in $\Omega$. Then 
\[
\frac{v(W_x,r)}{v(W_x,R)}\geq \left(\frac{\s_{K/N}(r)}{\s_{K/N}(R)}\right)^{N}\quad \text{and}\quad \frac{s(W_x,r)}{s(W_x,R)}\geq \frac{\int_0^r \s_{K/N}(t)^N\ dt}{\int_0^R \s_{K/N}(t)^N\ dt}
\]
for all $0<r\leq R\leq \mathrm{rad}_x(W_x):=\sup\{\d(x,y):y\in W_x\}$
\end{theorem}
\begin{proof}
Let $0<r\leq R\leq \mathrm{rad}_x(W_x)$, $0<\epsilon\leq \epsilon_0$ and $\eta>0$ and let 
\[
A_0 := B_{\epsilon} (x)\cap W_x\quad \text{and}\quad A_1 := (\overline{B}_{(1+\eta)R} (x)\setminus B_{R} (x))\cap W_x.
\]
In particular, the $(r/R)$-intermediate set $A_{r/R}$ between $A_0$ and $A_1$ is contained in $\Omega$, so we can apply Theorem \ref{thm:brunn-minkowski} with $K=0$ to get
\[
\m(A_{r/R})^{1/N} \geq (1-r/R)\m(A_0)^{1/N}+(r/R)\m(A_1)^{1/N}.
\]
We can take $\epsilon \to 0$, which yields
\[
\m((\overline{B}_{(1+\eta)r}(x)\setminus B_r(x))\cap W_x)\geq (r/R)^{N}\m((\overline{B}_{(1+\eta)R} (x)\setminus B_{R} (x))\cap W_x).
\]
We thus can conclude just as in the proof of \cite[Theorem 2.3]{sturm2}. 
\end{proof}

\section{Proofs}
\label{sec:proofs}

In this section we proceed to prove Theorem \ref{thm:principal}. The proof follows almost verbatim the arguments in \cite{liu2}. For the convenience of the reader, we elaborate on this argument and stress the needed changes due to the more general hypotheses. We proceed with the following technical Lemma assuming the hypothesis of Theorem \ref{thm:principal}. 

\begin{lemma}\label{lem:delta}
Assume that $p\in X$ and $R>0$ are such that $B_{2R}(p)\subset X\setminus B_{2D}(p_0)$. Then for every $m>0$ there exists $\delta=\delta(m) \in (0,1)$ such that whenever a subset $W\subset B_R(p)$ satisfies
\begin{equation}\label{eq:delta}
\m(W) \geq \frac{1}{m}\cdot \m(B_R(p))
\end{equation}
then there exists $q\in W$ such that $\d(q,p) \leq \delta R$, and in particular $B_{(1-\delta)R}(q)\subset B_R(p)$.
\end{lemma}

\begin{proof} 
Let $\delta\in(0,1)$ and $W\subset B_R(p)\setminus\overline{B}_{\delta R}(p)$ be such that \eqref{eq:delta} holds. Since $B_{2R}(p)\cap B_{2D}(p_0)=\varnothing$ then the optimal plan between any two probability measures supported on $B_{(1+\eta)R}(p)$ for sufficiently small $\eta>0$ is concentrated in geodesics outside $B_{D}(p_0)$. Therefore, applying Theorem \ref{thm:bishop-gromov} yields
\begin{align*}
\frac{1}{m}\leq \frac{\m(W)}{\m(B_R(p))}\leq \frac{\m(B_R(p)\setminus B_{\delta R}(p))}{\m(B_R(p))}\leq \frac{\int_{\delta R}^R t^{N-1}\ dt}{\int_0^{R} t^{N-1}\ dt} = 1-\delta^N,
\end{align*}
that is, $\delta\leq (1-1/m)^{1/N}$. Thus, if we take $\delta = (1-1/(2m))^{1/N}$ and any $W$ satisfying \eqref{eq:delta}, then $W\cap \overline{B}_{\delta R}(p) \neq \varnothing$, i.e there is some $q\in W$ such that $\d(q,p)\leq \delta R$. In particular, for such $q\in W$ and any $x\in B_{(1-\delta)R}(q)$ we have
\[
\d(x,p) \leq \d(x,q)+\d(q,p) < (1-\delta)R+\delta R = R,
\] so $B_{(1-\delta)R}(q)\subset B_R(p)$.
\end{proof}

\begin{proof}[Proof of Theorem \ref{thm:principal}]
Clearly we can assume that $K<0$. Moreover, by rescaling the metric in $X$ by $\sqrt{-K}$, we can further assume that $(X,\d,\m)$ satisfies $\CD(-1,N)$ condition. In particular, we get that $B\subset B_{D}(p_0)$ where $D = \sqrt{-K}D_0$.

For $\mu>2$ the result follows from the fact that $\ov{B}_{r}(p_0)\subset B_{\mu\cdot r}(p)$ for any $p\in\ov{B}_{r}(p_0)$, so we can set $C(N,KD_0^2,\mu)=1$. Therefore, we will assume from now on that $0<\mu\leq 2$. 

We now divide $S$ into the union $S_1$ and $S_2$ where
\[
S_1 = S\cap B_{\mu r/2}(p_0),\quad S_2 = S\setminus S_1. 
\]
If $S_1\neq \varnothing$ then it can be covered by just one $B_{\mu r}(p)$ with $p\in S_1$. In any case, we only need to estimate the covering number of $S_2$. We will actually estimate the number of $(\mu r/4)$-balls needed to cover $S_2$, so for simplicity let us denote $t=\mu/4$.

Now, fix some $\lambda >2$. The case when $tr\leq \lambda  D$  follows exactly as in \cite[Page 11]{liu2} (where instead of $\lambda$ it suffices to consider $2$). Therefore we assume that $tr > \lambda D$. In particular, for $q \in S_2$,
\[
B_{tr}(q)\cap B_{\lambda D}(p_0) = \varnothing.
\]

Write $\partial B_{\lambda D}(p_0)$ as the union of subsets $\{U_1,\dots,U_m\}$ such that $\d(x,y)<2D$ for any $x,y\in U_a$. This can be done as follows. Take a maximal set of points $\{q_1,\dots,q_m\} \subset \partial B_{\lambda D}(p_0)$ such that $\d(q_a,q_b)\geq D$, $a\neq b$. Then
\begin{equation*}
\begin{gathered}
\partial B_{\lambda D}(p_0)\subset \bigcup_{j=1}^{m}B_{D}(p_0),\\
B_{D/2}(q_a)\cap B_{D/2}(q_b) = \varnothing,\quad a\neq b.
\end{gathered}
\end{equation*}

Suppose $B_{D/2}(q_s)$ has the smallest volume among all $B_{D/2}(q_j)$. Since $\bigcup_{i=1}^{m}B_{D/2}(q_j)\subset B_{\left(1/2+2\lambda \right)D}(q_s)$, Bishop-Gromov comparison corresponding to the condition $\CD(-1,N)$ yields
\begin{equation}\label{eq:bound for m}
m<\frac{V^{-1}((1/2+2\lambda)D)}{V^{-1}(D/2)}.
\end{equation}
Note that the right-hand side of \eqref{eq:bound for m} depends only on $n$ and $D$. We define
\[
U_a = B_{D}(q_a) \cap \partial B_{\lambda  D}(p_0), \quad a =1,\dots, m.
\]

Let $M_r$ be the subset of $M$ consisting of all points on any minimal geodesic emanating from $p_0$ that is no shorter than $r$. Note that $M-B_{r}(p_0)\subset M_r$, and $M_r$ is star-shaped at $p_0$.

We now divide $M_{\lambda  D}$ into $m$ cones $K_a$ by defining $K_a$ to be the subset consisting of all points on any minimal geodesic emanating from $p_0$ that intersects $U_a$. Observe that, by the triangle inequality, if $d(x_i,p_0)>\lambda  D$, $x_i\in K_a$, $i = 1,2$, then any minimal geodesic connecting $x_1$ and $x_2$ will not pass through $B_{\lambda  D/2}(p_0)$. Indeed, let $\gamma_i$ be a minimal geodesic from $p_0$ to $x_i$ with $\gamma_i(\lambda D) \in U_a$, $i =1,2$. Then the broken geodesic from $x_1$ to $\gamma_1(\lambda D)$ to $\gamma_2(\lambda D)$ to $x_2$ has length no greater than
\[
d(x_1,p_0)+d(x_2,p_0)-2\lambda D+2D.
\]
On the other hand, if a minimal geodesic connecting $x_1$ and $x_2$ intersects $B_{\lambda  D/2}(p_0)$, then it would have a length greater than
\[
d(x_1,p_0)+d(x_2,p_0)-\lambda D,
\] which is a contradiction.

Now we can estimate the covering number just as in \cite{liu1,liu2}. For the  convenience of the reader, we will repeat some of the constructions.

Take a maximal set of points $\{p_1,\dots, p_k\}$ in $S_2$ such that $\d(p_i,p_j)>tr$, $i\neq j$. Then
\begin{equation*}
\begin{gathered}
S_2\subset \bigcup_{j} B_{tr}(p_j),\\
B_{tr/2}(p_i)\cap B_{tr/2}(p_j) = \varnothing,\quad i \neq j.
\end{gathered}
\end{equation*}
We then divide the points $p_j$ into $m$ families as follows: for each ball $B_{tr/2}(p_j)$, look at $\m(B_{tr/2}(p_j)\cap K_a)$, $a =1,\dots, m$. Fix an $a_j$ such that $\m(B_{tr/2}(p_j)\cap K_{a_j})$ is maximal. Then
\begin{equation}
\m(B_{tr/2}(p_j) \cap K_{a_j}) \geq \frac{1}{m}\m(B_{tr/2}(p_j)).
\end{equation}
We denote
\[
B_{p_j}^{L,a_j}:= B_{tr/2}(p_j)\cap K_{a_j},
\]
and place $p_j$ in the $a_j$-th family, call it $\F_{a_j}$. Fix a $K_{a}$. Suppose $B^{L,a}_{p}$ has the smallest volume among all $B^{L,a}_{p_j}$ in this cone. By Lemma \ref{lem:delta}, we can find a $q\in B^{L,a}_p$ such that
\[
B_{(1-\delta)tr/2}(q)\subset B_{tr/2}(p).
\]
Let $W_q$ be the star-shaped set such that $y \in W_q$ if and only if there is a point $x$ belonging to either $B_{(1-\delta)tr/2}(q)$ or $B^{L,a}_{p_j}$ for some $p_j\in \F_a$ and there is a minimal geodesic $\gamma$ connecting $q$ and $x$ which passes $y$. 

Observe that for $\epsilon_0 = \min\{(1-\delta)tr,(\lambda -2)D/2\}>0$, $z\in B_{\epsilon_0}(q)$ and $y\in W_q$, any geodesic joining $z$ with $y$ is outside $B_{\lambda D/2}(p_0)$. Indeed, if $x\in B_{(1-\delta)tr/2}(q)\cup \bigcup B^L_{p_j}$ is such that $y$ is in a geodesic $\gamma$ joining $q$ with $x$ and $\gamma_1$ is a geodesic joining $z$ with $y$ and passing through $B_{\lambda  D/2}(p_0)$ then the broken geodesic from $q$ to $z$ to $y$ to $x$ will have length greater than 
\[
\d(q,p_0)+\d(x,p_0)-\lambda D.
\]
However, it also has length no greater than
\[
\epsilon_0+\epsilon_0+\d(q,y)+\d(y,x)=2\epsilon_0+L(\gamma)\leq 2\epsilon_0+\d(q,p_0)+\d(x,p_0)-2\lambda D+2D
\]
which is a contradiction.

By a simple triangle inequality we get that $\d(q,y)\leq (2+t)r$ for all $y\in W_q$. Therefore, applying Theorem \ref{thm:bishop-gromov} with $K=0$, we get
\[
\frac{\m(W_q)}{\m(B_{(1-\delta)tr/2}(q))}\leq 2^N(2+t)^N(1-\delta)^{-N}t^{-N}
\] However,
\[
\frac{\m(W_q)}{\m(B_{(1-\delta)tr/2}(q))}\geq \sum_{p_j\in \F_a}\frac{\m(B^{L,a}_{p_j})}{\m(B_{tr/2}(p))}\geq \frac{\# \F_a}{m}
\] thus we get
\[
\#\F_a \leq 2^N(2+t)^N(1-\delta)^{-N}t^{-N}m.
\] Adding up the contributions from the $m$ families $\F_a$, we get that 
\begin{equation}\label{eq:k}
k \leq 2^N(2+t)^N(1-\delta)^{-N}t^{-N}m^2.
\end{equation}
The right hand side of \eqref{eq:k} depends on $N$, $\mu$ and $m$. However $m$ is bounded above by the right hand side of \eqref{eq:bound for m}, which is a function depending on $KD_0^2$ and $\lambda$, increasing with respect to $\lambda$. Taking $\lambda\searrow 2$, we get the required constant $C= C(N,KD_0^2,\mu)>0$.  
\end{proof}

\begin{definition}
\label{def:ends}
Let $(X, \d)$ be a metric space and $k \in \N$. We say that $X$ has $k$ ends if both the following are true:
\begin{enumerate}
\item for any $K$ compact, $X \setminus K$ has at most $k$ unbounded connected components,
\item there exists $K'$ compact such that $X \setminus K'$ has exactly $k$ unbounded connected components.
\end{enumerate}
\end{definition}

\begin{proof}[Proof of Corollary \ref{cor:principal}] 
If the result does not hold true, we take $r$ large enough so that $X\setminus \overline{B}_r(p_0)$ has $n > C(N,KD_0^2,1/2)$ unbounded connected components. It is clear that each such unbounded connected component $E$ requires at least one ball of radius $r$ to cover $E\cap \partial \overline{B}_{2r}(p_0)$. This contradicts Theorem \ref{thm:principal}.
\end{proof}

%\bibliographystyle{abbrv}  

%\begin{thebibliography}{1}
%\end{thebibliography}

\end{document}